\documentclass{article}
\usepackage{amsmath,amsthm,amssymb,graphicx}

\usepackage{faktor}

\newtheorem{theorem}{Theorem}[section]
\newtheorem{lemma}[theorem]{Lemma}
\newtheorem{conjecture}[theorem]{Conjecture}

\theoremstyle{remark}
\newtheorem{remark}[theorem]{Remark}

\begin{document}
 
\title{On the topological Kalai-Meshulam conjecture}

\author{Alexander Engstr\"om\footnote{{\tt alexander.engstrom@aalto.fi}}  }

\date\today

\maketitle
 
\begin{abstract}
Chudnovsky, Scott, Seymour and Spirkl recently proved a conjecture by Kalai and Meshulam stating that the reduced Euler characteristic of the independence complex of a graph without induced cycles of length divisible by three is in $\{-1,0,1\}.$ Gauthier had earlier proved that assuming no cycles of those lengths, induced or not. Kalai and Meshulam also stated a stronger topological conjecture, that the total betti numbers are in $\{0,1\}.$ Towards that we prove an even stronger statement in the same setting as Gauthier: The independence complexes are either contractible or homotopy equivalent to spheres. We conjecture that it also holds in the general setting.
\end{abstract}

\section{Introduction}

Kalai and Meshulam \cite{KalaiBlog} made a sequence of conjectures relating topological and chromatic properties of graphs about two decades ago. The independence complex $\mathrm{Ind}(G)$ of a graph $G$ is the simplicial complex formed by its independent sets. The reduced Euler characteristic is $\tilde{\chi}(\mathrm{Ind}(G))=$
$ \sum_i (-1)^i \beta_i(\mathrm{Ind}(G)) = -\sum_i (-1)^i \,\, \# \textrm{independent sets of order $i$ in G;}$ and the total Betti number is $\beta(\mathrm{Ind}(G))=  \sum_i \beta_i(\mathrm{Ind}(G)).$

\begin{conjecture}\label{c1}
A graph with high enough chromatic number has an induced cycle of length divisible by three.
\end{conjecture}

\begin{conjecture}\label{c2}
For any graph $G,$ we have that $-1 \leq \tilde{\chi}(\mathrm{Ind}(H)) \leq 1$ for every induced subgraph $H$ of $G$ if and only if there are no induced cycles of length divisible by three in $G.$
\end{conjecture}

\begin{conjecture}\label{c3}
For any graph $G,$ we have that $\beta(\mathrm{Ind}(H)) \leq 1$ for every induced subgraph $H$ of $G$ if and only if there are no induced cycles of length divisible by three in $G.$
\end{conjecture}

Conjecture~\ref{c1} was proved by Bonamy, Charbit, and Thomass\'e \cite{BCT14}, and generalised by Scott and Seymour \cite{SS17}.
Towards Conjecture~\ref{c2}, Gauthier~\cite{Gau17} proved in his PhD thesis at Princeton that if there are no cycles of length divisible by three in $G,$ then
$-1 \leq \tilde{\chi}(\mathrm{Ind}(G)) \leq 1.$ The full Conjecture~\ref{c2} was recently proved by Chudnovsky, Scott, Seymour, and Spirkl~\cite{CSSS18}. So far Conjecture~\ref{c3} is open, but our main result is a small step towards it. Employing the graph theoretic characterisation by Gauthier we prove that:
\begin{theorem}\label{thm}
If $G$ is a graph without cycles of length divisible by three, then $\mathrm{Ind}(G)$ is contractible or homotopy equivalent to a sphere.
\end{theorem}
We propose that the sequence of conjectures by Kalai and Meshulam can be amended by:
\begin{conjecture}\label{c4}
For any graph $G,$ we have that $\mathrm{Ind}(H)$ is contractible or homotopy equivalent to a sphere for every induced subgraph $H$ of $G$ if and only if there are no induced cycles of length divisible by three in $G.$
\end{conjecture}
Our theorem generalises one by Ehrenborg and Hetyei \cite{EhrHet06} that shows that for any forest its independence complex is contractible or homotopy equivalent to a sphere.

Some historical remarks on this setup might be warranted. When Lov\'asz~\cite{Lov78} proved Kneser's conjecture he demonstrated that the topology of polyhedral complexes constructed from graphs can inform us about their chromatic numbers. The independence complex is essentially an easily defined piece of one of Lov\'asz complexes, see the spectral sequences by Babson and Kozlov in Sections 4 and 5 of \cite{BK07} for how it fits in technically. But as the independence complex of a graph with an isolated vertex is a cone, the topology of the complex by itself doesn't provide any chromatic information. The conjectures above by Kalai and Meshulam proposed that one should consider the topology of all induced subgraphs at once. That package of information makes sense, because the dimensions of those homologies are the fine graded betti numbers of the variety cut out by $x_ux_v$ for all edges $uv$ of the graph, a fundamental set of invariant in algebraic geometry describing projective resolutions. For cycle graphs the homology depends on their length, as the independence complex of a cycle of length $l$ is homotopy equivalent to a wedge of two spheres if $3$ divides $l$, and homotopy equivalent to one sphere otherwise. To minimise the amount of homology to tentatively restrict the chromatic number, while having some cycles in the graphs, Kalai and Meshulam conditioned on that no cycles of length divisible by three would be allowed.

Further evidence towards that Conjecture~\ref{c4} is reasonable, is in Engstr\"om~\cite{E14} where the total Betti number is bounded in terms of the number of disjoint cycles in a graph. Another indirect reason is the structure of the proof of Conjecture~\ref{c2} by Chudnovsky, Scott, Seymour, and Spirkl~\cite{CSSS18}. As their topic is the reduced Euler charateristic, they have lots of powers of $-1$ to keep track of due to dimension jumps. To get rid of that, they introduce \emph{counters} that actually seems to count spheres, although they don't claim that. Unfortunately it is difficult to replace the inclusion/exclusion formulas of that paper with suitable spectral sequences that would provide a proof on the homological level, not to mention the diagrams of spaces that would give the homotopy theoretic results.

To explain another interesting connection we start off with an example. The graph $K_{n,n}$ have no induced cycles of length divisible by three. Consider the simplex $\Delta \subset \mathbb{R}^n$ whose vertices are the $n$ unit vectors and define the polytope $P=\Delta \times \Delta \subset \mathbb{R}^{2n}.$ The vertices of $P$ correspond to the edges of $K_{n,n}.$ One can see that the faces of $P$ are in bijection with the induced subgraphs $H$ of $K_{n,n}$ such that $\mathrm{Ind}(H)$ is homotopy equivalent to a sphere (in this case equivalently that $\tilde{\chi}(\mathrm{Ind}(H)) = \pm 1$ for those that prefer a completely combinatorial statement). Thus, in this case, the fine structure of what induced subgraphs that are contractible or spheres up to homotopy are encoded by the edge polytope of $K_{n,n}$ (see Tran and Ziegler~\cite{TZ14} for a introduction to these polytopes). The connection between the Kalai-Meshulam conjectures and the edge polytopes is implicit in a paper by N\'oren~\cite{N19}. The projective resolutions mentioned earlier can be constructed as chain complexes supported by regular cell complexes, a concept called cellular resolutions introduced by Bayer and Sturmfels~\cite{BS98}. Nor\'en provided a linear algebra type condition on the face structure of edge polytopes to confirm that the related minimal regular cellular resolutions are unique, and that condition at the same time confirms that the fine graded betti numbers are as those considered by Kalai and Meshulam.
\begin{conjecture}\label{c5}
For any graph $G$ without induced cycles of length divisible by three, there is a polytope $P$ whose vertices are labelled by the edges of $G$ and its faces are in bijection with the induced subgraphs of $G$ whose reduced Euler characteristic is $\pm 1.$  (That is, the vertices of a face of $P$ correspond to the edges of an induced subgraph of $G$.)
\end{conjecture}
The edge polytopes do not solve the conjecture alone. If $G$ is the five cycle then $P$ can be a regular pentagon in the plane, but the edge polytope of $G$ is a simplex.

\section{Independence complexes}

Some basic graph notions. The \emph{length} of a path is its number of edges. The \emph{neighbourhood} of $v$ in $G$ is $N_G(v)=\{u \in V(G) \mid uv \in E(G) \},$ and $N_G[v] = \{v\} \cup N_G(v).$ When $G$ is evident from the context, it is dropped. The \emph{contraction}, or identification, of a subgraph $H$ of $G$ to a new point is denoted by $G / H;$ the \emph{deletion} by $G \setminus H.$ A connected graph on at least three vertices is 2--connected if no deletion of a vertex disconnects it. For basic combinatorial topology, in particular related to complexes from graphs, we refer to Jonsson~\cite{Jon08}.

\begin{remark}
In Lemma~\ref{lem:path} we prove that if an induced path of length three is contracted to a point, then the corresponding independence complexes are related up to homotopy by a suspension. A weaker version, where an induced path of length four is contracted to an edge, is less technical to prove and it was first done by Csorba~\cite{Cso09}. The question of subdividing edges of graphs and its influence on the topology of the associate independence complexes partly originate from theoretical physics, where results in that manner were proved using Hodge theory, see \cite{Eng09b,HS10}. Adamaszek~\cite{Ada12}, Barmak~\cite{Bar13} and Jonsson~\cite{Jon12} introduced interesting techniques to decompose independence complexes of graphs, in particular in the triangle-free setting. The lemmas in this paper can be employed to simplify calculations of the extreme Khovanov cohomology of knots, see Gonz\'alez-Meneses, Manch\'on, and Silvero~\cite{GMMS18}; and Przytycki and Silvero~\cite{PS18}. There is a  shorter proof of Lemma~\ref{lem:path} by diagrams of spaces as surveyed by Welker, Ziegler, and \v{Z}ivaljevi\'c~\cite{WZZ99}, and the same goes for the remaining lemmas, but as that might be considered slightly too technical for our purposes, we have not employed it.
\end{remark}

\begin{lemma}\label{lem:path}
If $G$ is a graph with a path $P$ of length three whose internal vertices are of degree two and whose end vertices are not adjacent, then
$\mathrm{Ind}(G)$ is homotopy equivalent to $\mathrm{susp}(\mathrm{Ind}(G / P)).$
\end{lemma}

\begin{proof}
Denote the vertices of the path by $abcd.$ The proof goes by constructing a sequence of homotopy equivalent spaces. First some faces are to be collapsed away. Let
\[
\mathcal{M} = \left\{
(\sigma, \sigma \cup \{c \} ) 
\left|
\begin{array}{ll}
\sigma \in \mathrm{Ind}(G) &  b,c,d \not\in \sigma \\
a \in \sigma & \sigma \cap N(d) \neq \emptyset  \\
\end{array}
\right.
\right\}
\]
be a matching of faces of $\mathrm{Ind}(G).$ To collapse away $\mathcal{M}$ we should show that the remaining faces form a simplicial complex. That is, if $(\sigma, \sigma \cup \{c \} ) \in \mathcal{M}$ and $\sigma \subseteq \tau \in \mathrm{Ind}(G)$ with $c \not\in \tau,$ then $(\tau,  \tau \cup \{c \} ) \in \mathcal{M}.$ The conditions to be in $\mathcal{M}$ are checked:
\begin{itemize}
\item[-] $\tau \in \mathrm{Ind}(G)$ and $c \not\in \tau$ by assumption,
\item[-] $a \in \tau$ as $a\in \sigma \subseteq \tau,$
\item[-] $b \not\in \tau$ as $a\in \tau$ and $\tau \in \mathrm{Ind}(G),$
\item[-] $\tau  \cap N(d) \neq \emptyset$ as $\sigma \cap N(d) \neq \emptyset$ and $\sigma \subseteq \tau,$ and
\item[-] $d \not\in \tau$ as $\tau  \cap N(d) \neq \emptyset$ and $\tau \in \mathrm{Ind}(G).$
\end{itemize}
Thus, removing the faces fo $\mathcal{M}$ from $\mathrm{Ind}(G)$ is a deformation retraction to a homotopy equivalent complex. The symmetric case
\[
\mathcal{N} = \left\{
(\sigma, \sigma \cup \{b \} ) 
\left|
\begin{array}{ll}
\sigma \in \mathrm{Ind}(G) &  a,b,c \not\in \sigma \\
d \in \sigma & \sigma \cap N(a) \neq \emptyset  \\
\end{array}
\right.
\right\}
\]
also yields a deformation retraction. The presence of the vertices $a$ or $d$ shows that no face is in both $\mathcal{M}$ and $\mathcal{N},$ and we can remove both of the matchings from $\mathrm{Ind}(G)$ to get a homotopy equivalent simplicial complex denoted $\Delta.$ The changes are parsed out in detail in this table:
\[
\begin{array}{c|cc}
s & \{ \sigma | \sigma \cup s \in  \mathrm{Ind}(G) ,  \sigma \cap P = \emptyset  \} & \{ \sigma | \sigma \cup s \in \Delta,  \sigma \cap P = \emptyset  \} \\
\hline
\emptyset & \mathrm{Ind}(G \setminus P) &   \mathrm{Ind}(G \setminus P) \\
\{a\} & \mathrm{Ind}(G \setminus (P \cup N(a)) )  &    \mathrm{Ind}(G \setminus (P \cup N(a) \cup N(d) ) ) \\
\{b\} & \mathrm{Ind}(G \setminus P) &   \mathrm{Ind}(G \setminus P) \\
\{c\} & \mathrm{Ind}(G \setminus P) &   \mathrm{Ind}(G \setminus P) \\
\{d\} &  \mathrm{Ind}(G \setminus (P \cup N(d)) ) &     \mathrm{Ind}(G \setminus (P \cup N(a) \cup N(d) ) ) \\
\{a,c\} &  \mathrm{Ind}(G \setminus (P \cup N(a)) ) &    \mathrm{Ind}(G \setminus (P \cup N(a) \cup N(d) ) )  \\
\{a,d\} &  \mathrm{Ind}(G \setminus (P \cup N(a) \cup N(d) ) ) &    \mathrm{Ind}(G \setminus (P \cup N(a) \cup N(d) ) )  \\
\{b,d\} &  \mathrm{Ind}(G \setminus (P \cup N(d)) )   &    \mathrm{Ind}(G \setminus (P \cup N(a) \cup N(d) ) ) \\
\end{array}
\]
Let $p$ be the point in $G / P$ that the path $P$ is contracted to, and $bc$ be the graph on the two adjacent vertices $b$ and $c.$ According to the
table above $\Delta = \Delta_1 \cup \Delta_2$ where
\[ \Delta_1 =  \textrm{Ind}(bc) \ast \textrm{Ind}(G \setminus P)  =   \textrm{Ind}(bc) \ast \textrm{Ind}( (G/P)\setminus p ) \]
and
\[ \Delta_2 = \textrm{Ind}(P) \ast \textrm{Ind}(G \setminus  (P \cup N(a) \cup N(d) ) ) =   \textrm{Ind}(P) \ast \textrm{Ind}( (G/P)  \setminus   N_{G/P}[p]  ). \]
The subcomplex $\Delta_2$ is contractible as $\textrm{Ind}(P)$ is contractible. Contracting it, we get that $\Delta$ is homotopy equivalent to
\[
\faktor{\Delta}{\Delta_2} = \faktor{\Delta_1 \cup \Delta_2}{\Delta_2} = \faktor{\Delta_1}{\Delta_1 \cap \Delta_2} = \faktor{\Delta_1 \cup p\ast (\Delta_1 \cap \Delta_2)}{p \ast (\Delta_1 \cap \Delta_2).}
\]
As the cone $p \ast (\Delta_1 \cap \Delta_2)$ is contractible, the rightmost space is homotopy equivalent to $\Delta_1 \cup p\ast (\Delta_1 \cap \Delta_2).$ But that is nothing else than $\textrm{Ind}(bc) \ast \textrm{Ind}( G/P )$ as its deletion of $p$ is $\Delta_1$ and its link of $p$ is $\Delta_1 \cap \Delta_2.$ We have demonstrated that 
$\textrm{Ind}( G )$ is homotopy equivalent to $\textrm{Ind}(bc) \ast \textrm{Ind}( G/P ),$ the suspension of $\textrm{Ind}( G/P ).$
\end{proof}

\begin{lemma}\label{lem:fold}[Engstr\"om, Lemma 3.2. in \cite{Eng09a}]
If $G$ is a graph with two distinct vertices $u$ and $v$ satisfying $N(u) \subseteq N(v),$ then $\textrm{Ind}( G )$ and $\textrm{Ind}( G \setminus v )$ are homotopy equivalent.
\end{lemma}

\begin{lemma}\label{lem:square}
If $G$ is a graph with a path $P$ of length three whose internal vertices are of degree two and whose end vertices are adjacent, then
$\mathrm{Ind}(G)$ is homotopy equivalent to $\mathrm{susp}(\mathrm{Ind}(G \setminus P)).$
\end{lemma}
\begin{proof}
Denote the path by $abcd$ and note that $N_G(c) \subseteq N_G(a).$ By Lemma~\ref{lem:fold}, $\textrm{Ind}( G )$ and $\textrm{Ind}( G \setminus a)$ are homotopy equivalent. As $N_{G \setminus a}(d) \subseteq N_{G \setminus a}(b),$ the same lemma gives that $\textrm{Ind}( G \setminus \{a,d\})$ and $\textrm{Ind}( G )$ are homotopy equivalent. The homotopy equivalence with $\mathrm{susp}(\mathrm{Ind}(G \setminus P))$ follows from that $G \setminus \{a,d\}$ is the disjoint union of $G\setminus P$ and the edge $bc.$
 \end{proof}

This is an old glueing lemma. For a more general result in the independence complex setting, see Proposition 3.1 in Adamaszek~\cite{Ada12}.

\begin{lemma}\label{lem:glue}
Let $\Delta$ be a simplicial complex with a vertex $v.$ If $\Delta \setminus v$ is contractible, then $\Delta$ is homotopy equivalent to 
$\mathrm{susp}(\mathrm{lk}_\Delta(v)).$ If $\mathrm{lk}_\Delta(v)$ is contractible, then $\Delta$ is homotopy equivalent to $\Delta \setminus v.$
\end{lemma}

\begin{lemma}\label{lem:separate}
Let $G$ be a graph with a vertex $u$, $H$ be a graph with a vertex $v,$ and $G \vee H$ be their disjoint union with $u$ and $v$ identified.
If the independence complexes of $G, G\setminus u, G \setminus N_G[u], H, H\setminus v, H \setminus N_H[v]$ are homotopy equivalent to spheres or contractible, then the independence complex of $G \vee H$ is homotopy equivalent to a sphere or contractible.
\end{lemma}
\begin{proof}
Denote the vertex in $G \vee H$ that $u$ and $v$ was identified with by $w.$ One of the complexes $\mathrm{Ind}( G), \textrm{Ind}( G\setminus u )$ and $\textrm{Ind}(G \setminus N_G[u] ) $ has to be contractible, because if they all were spheres the Mayer-Vietoris sequence for $\mathrm{Ind}( G) = \textrm{Ind}( G\setminus u ) \cup u \ast \textrm{Ind}(G \setminus N_G[u] )$ would not be exact. 

Case 1: $\textrm{Ind}( G\setminus u )$  is contractible.
The deletion of $w$ from $\textrm{Ind}( G\vee H )$ is contractible as
\[   \textrm{Ind}( G\vee H  \setminus w) = \textrm{Ind}( (G \setminus u) \sqcup (H \setminus v) ) =  \textrm{Ind}( G \setminus u) \ast  \textrm{Ind}  (H \setminus v). \]
By Lemma~\ref{lem:glue} the complexes $\textrm{Ind}( G\vee H )$ and $\mathrm{susp}(\mathrm{lk}_{\textrm{Ind}( G\vee H )}(w))$ are homotopy equivalent. By assumption
the independence complexes of $G \setminus N_G[u]$ and $H \setminus N_H[v]$ are homotopy equivalent to spheres or contractible. As
\[ \mathrm{lk}_{ \textrm{Ind}( G\vee H )}(w) = \mathrm{lk}_{ \textrm{Ind}( G)}(u ) \ast \mathrm{lk}_{ \textrm{Ind}( H)}(v )
 = \textrm{Ind}(G \setminus N_G[u] ) \ast \textrm{Ind}(H \setminus N_H[v] ),  \]
the independence complex of $G \vee H$ is homotopy equivalent to a sphere or contractible.

Case 2: $\textrm{Ind}(G \setminus N_G[u] )$ is contractible. It follows that $\mathrm{lk}_{ \textrm{Ind}( G\vee H )}(w)$ is contractible, and by Lemma~\ref{lem:glue}, the independence complex of $G \vee H$ is homotopy equivalent to the join of  $\textrm{Ind}( G\setminus u )$ and $\textrm{Ind}( H\setminus v ),$ and they are homotopy equivalent to spheres or contractible. 

Case 3: $\mathrm{Ind}( G)$ is contractible. Construct a graph $G \dot{\vee} H$ by adding a vertex $\dot{w}$ to $G \vee H$ whose neighbours are $N_{H}(v).$
The link of $\dot{w}$ in $\mathrm{Ind}( G \dot{\vee} H )$ is contractible as
\[
\mathrm{Ind}( G \dot{\vee} H  \setminus N_{G \dot{\vee} H}[\dot{w}])=\mathrm{Ind}( G \sqcup (H \setminus N_H[v])) = \mathrm{Ind}( G ) \ast  \mathrm{Ind} (H \setminus N_H[v]))
\]
and $\mathrm{Ind}( G )$ is contractible. It follows by Lemma~\ref{lem:glue} that $\mathrm{Ind}( G \dot{\vee} H )$ is homotopy equivalent to
\[ \mathrm{Ind}( G \dot{\vee} H ) \setminus \dot{w} = \mathrm{Ind}( G \vee H ). \]
The neighbourhood of $w$ includes the neighbourhood of $\dot{w}$ in $G \dot{\vee} H,$ and by Lemma~\ref{lem:fold}, $\mathrm{Ind}( G \dot{\vee} H )$ is homotopy equivalent to $\mathrm{Ind}( G \dot{\vee} H ) \setminus v.$ The graph $(G \dot{\vee} H) \setminus v$ is isomorphic to $(G \setminus u) \sqcup H$ by relabelling $\dot{w}$ to $v.$ Thus, $\mathrm{Ind}( G \dot{\vee} H )$ and $\mathrm{Ind}( G \vee H )$ are homotopy equivalent to
$ \mathrm{Ind}( G \setminus u ) \ast \mathrm{Ind}( H ).$ As $\mathrm{Ind}( G \setminus u ) $ and $\mathrm{Ind}( H )$ are homotopy equivalent to spheres or contractible, so is $\mathrm{Ind}( G \vee H ).$
\end{proof}

\begin{theorem}[Theorem 1.0.2, Gauthier \cite{Gau17}]~\label{thm:G}
If $G$ is a 2--connected graph with no cycles of length divisible by three, then there are two adjacent vertices of degree two or two non-adjacent vertices of degree two with the same neighbourhood.  
\end{theorem}

We are now ready to prove our main result, Theorem~\ref{thm}.

\begin{proof} Clearly the subgraphs of $G$ also lack cycles of length divisible by three.

The proof is by induction on the number of vertices. If $G$ has less than three vertices we are done. If $G$ is not connected we are done, as its independence complex is the join of the independence complexes of the connected components of $G.$

If $G$ is connected but not 2--connected, then there is a vertex whose removal disconnects it. By Lemma~\ref{lem:separate} we are done.

If $G$ is 2--connected we employ the characterisation of Theorem~\ref{thm:G}. 

First the case of two adjacent vertices $bc$ of degree two on a path $P=abcd$ in $G.$ If the vertices $a$ and $d$ are non-adjacent, then $\mathrm{Ind}(G)$ is homotopy equivalent to $\mathrm{susp}(\mathrm{Ind}(G / P))$ according to Lemma~\ref{lem:path}. There can be no cycles in $G / P$ of length divisible by three, as there are none in $G$ and $P$ is of length three. We may apply induction as $G/P$ has fewer vertices than $G.$
If the vertices $a$ and $d$ are adjacent, then $\mathrm{Ind}(G)$ is homotopy equivalent to $\mathrm{susp}(\mathrm{Ind}(G \setminus P))$ by Lemma~\ref{lem:square} and we are done by induction. 

The second case in the characterisation is that there are two distinct non-adjacent vertices $u$ and $v$ with the same neighbourhood in $G.$ By Lemma~\ref{lem:fold},
$\mathrm{Ind}(G)$ is homotopy equivalent to $\mathrm{Ind}(G \setminus v)$ and induction applies. 
\end{proof}

\end{document}